\documentclass{amsart}
\usepackage{tikz, tikz-cd}
\usetikzlibrary{calc,decorations.markings,decorations.pathmorphing,arrows}
\usetikzlibrary{shapes.misc}
\usepackage{amsmath, amssymb, amsthm, amsfonts, mathrsfs, amsfonts}
\usepackage[a4paper,left = 2.5cm, right= 2.5cm, top = 2.5cm, bottom = 2.5cm, marginpar=2cm]{geometry}
\usepackage{xcolor}
\usepackage[all,tips]{xy}
\usepackage{hyperref}

\newcommand{\Z}{\ensuremath{\mathbb{Z}}}

\newtheorem{thm}{Theorem}[section]

\newtheorem*{question*}{Question}

\newtheorem{cor}[thm]{Corollary}
\newtheorem{lemma}[thm]{Lemma}
\newtheorem{prop}[thm]{Proposition}
\theoremstyle{definition}
\newtheorem{defin}[thm]{Definition}

\newtheorem{nota}[thm]{Notation}
\newtheorem{rmk}[thm]{Remark}

\newcommand{\on}{\operatorname}

\newcommand{\sh}{\ensuremath{\mathbf{s}}}
\newcommand{\cyc}{\ensuremath{\on{cyc}}}
\newcommand{\source}{\ensuremath{\mathfrak{s}}}
\newcommand{\target}{\ensuremath{\mathfrak{t}}}

\newcommand{\cC}{\mathcal{C}}
\newcommand{\cD}{\mathcal{D}}
\newcommand{\fd}{\operatorname{fd}}
\newcommand{\perf}{\operatorname{per}}

\title{Quivers with potentials and actions of finite abelian groups}

\author{Simone Giovannini}
\address{ORACLE Italy,
	Viale Fulvio Testi 136,
	20092 Cinisello Balsamo (Milano),
	Italy}
\email{simone.giovannini@oracle.com}

\author{Andrea Pasquali}
\address{Universit\"at Stuttgart,
	Institut f\"ur Algebra und Zahlentheorie,
	Pfaffenwaldring 57
	D-70569 Stuttgart,
	Germany}
\email{andrea.pasquali@mathematik.uni-stuttgart.de}

\author{Pierre-Guy Plamondon}
\address{Universit\'e Paris-Saclay, CNRS, Laboratoire de math\'ematiques d'Orsay, 91405, Orsay, France}
\email{pierre-guy.plamondon@math.u-psud.fr}
\begin{document}

\begin{abstract}
Let $G$ be a finite abelian group acting on a path algebra $kQ$ by permuting the vertices and preserving the arrowspans. Let $W$ be a potential on the quiver $Q$ which is fixed by the action. We study the skew group dg algebra $\Gamma_{Q, W}G$ of the Ginzburg dg algebra of $(Q, W)$. It is known that $\Gamma_{Q, W}G$ is Morita equivalent to another Ginzburg dg algebra $\Gamma_{Q_G, W_G}$, whose quiver $Q_G$ was constructed by Demonet. In this article we give an explicit construction of the potential $W_G$ as a linear combination of cycles in $Q_G$, and write the Morita equivalence explicitly.
As a corollary, we obtain functors between the cluster categories corresponding to the two quivers with potentials.
\end{abstract}

\maketitle

\section{Introduction}\label{sec:intro}
Quivers with potentials were introduced in \cite{DWZ08} as a tool to categorify the cluster algebras of S.~Fomin and A.~Zelevinsky \cite{FZ02}.  

A potential on a quiver is a (possibly infinite) linear combination of cyclic paths, considered up to cyclic equivalence.  A quiver with potential~$(Q,W)$ defines an associative algebra, the Jacobian algebra~$J(Q,W)$, which can be viewed as the cohomology in degree zero of a differential graded (=dg) algebra, the Ginzburg dg algebra~$\Gamma_{Q,W}$ \cite{Gin06,Ami09}.  

In this paper, we are interested in quivers with potentials with a finite group action.  Let~$(Q,W)$ be a quiver with potential and let~$G$ be a group acting on the path algebra~$kQ$ by sending vertices to vertices and fixing the space of all arrows; assume moreover that~$G$ fixes the potential~$W$. Such an action defines an action on the Jacobian algebra and on the Ginzburg dg algebra.  As in \cite{RR85} (extended to dg algebras), one can construct the skew group algebra~$J(Q,W)G$ and skew group dg algebra~$\Gamma_{Q,W}G$.  It was proved in~\cite{LM18} that the latter is Morita equivalent to the Ginzburg dg algebra of a quiver with potential; however, the proof does not readily give a way to compute it.

Our aim in this paper is to compute this quiver with potential $(Q_G, W_G)$ when the group~$G$ is abelian.  The quiver itself can be computed for any finite group using the work of \cite{Dem10}.  In the case where the group is of order~$2$ and acts on~$Q$, the potential $W_G$ was computed in~\cite{AP17}; this was used to describe the cluster category of a triangulated surface with punctures. More recently, $W_G$ was computed in \cite{GP19} for $G$ any cyclic group under some assumptions on the action.
Finally, an algorithm to compute $W_G$ for any finite group was given in \cite{LM18b}, but it relies on inverting a possibly large matrix and is thus not practical for computing examples.

We define $W_G$ by considering the image of $W$ via a naturally defined function $\iota: \widehat{k Q}\to (\widehat{k Q})G$ (which is not an algebra morphism). Our main result is Theorem~\ref{thm:main}, which claims that $W_G$ thus defined gives an explicit Morita equivalence between 
$\Gamma_{Q, W}G$ and $\Gamma_{Q_G, W_G}$. Moreover, in Proposition~\ref{prop:iota(path)} we present formulas which can be used in practice to express $W_G$ explicitly as a linear combination of cycles of $Q_G$.

As in \cite{AP17}, our results allow us to obtain functors between (generalized) cluster categories.  It is, however, unclear how the work of \cite{AP17} could be generalized to surfaces with orbifold points of order larger than 2. A similar issue arises in \cite{LFV17}, where triangulations of a disk with an orbifold point of order 3 give rise to algebras which are not quite cluster algebras.

We note that certain group actions on quivers with potentials are studied in \cite{PS17}.  Similar methods are also applied in \cite{AB19} in a different context, that is, the study of the derived category of skew-gentle algebras.  We also note that group actions on cluster algebras appear implicitly in the literature on cluster algebras from surfaces with orbifold points (see, for instance, \cite{FST12,FT17}); however, in these papers, the resulting cluster algebra is not of simply-laced type and could be said to be obtained by ``folding'' a quiver with a group action.  In this paper, we do not ``fold'' quivers with potential with a group action, and so all of the objects that we study are, in a sense, ``simply-laced''.

The paper is structured as follows.
In Section~\ref{sec:notation} we set up some conventions and recall the notions of skew group algebra and Ginzburg dg algebra. Sections~\ref{sec:sga} and \ref{sec:iso} are devoted to studying in depth the case of a skew group algebra of a path algebra by an abelian group.
Section~\ref{sec:mainresult} contains the definition of the potential $W_G$ and our main result, followed by a discussion of the consequences on generalized cluster categories. Finally, in Section~\ref{sec:example} we compute a detailed example. Section~\ref{sec:index_nota} contains an index of notation for the convenience of the reader.

\subsubsection*{Acknowledgement}
We are thankful to Patrick Le Meur for his encouragement and helpful discussions about this project. Most of this work was carried out while S.~G.~and A.~P.~were visiting the D\'epartement de Math\'ematiques d'Orsay, which we also thank. S.~G.~was supported by grants BIRD163492 and DOR1690814 of Padova University. A.~P.~was supported by Uppsala University and the Alexander von Humboldt Foundation. P.-G.~P.~was supported by French ANR grants SC3A (15CE40000401).

\section{Notation and conventions}\label{sec:notation}
In this section we fix some notation, as well as briefly recall the definitions of several objects we are going to discuss. Since this article is technical and the notation is quite heavy, we have included a table of symbols in Section~\ref{sec:index_nota}.

\subsection{Basics}\label{subsec:basics}
We fix an algebraically closed field $k$.
An \emph{algebra} means a finite-dimensional associative unital $k$-algebra. We call a \emph{basic version} of an algebra $A$ any basic algebra which is Morita equivalent to $A$.
A \emph{quiver} is a finite directed graph. For an arrow $a$ in a quiver, we write $\source(a)$ and $\target(a)$ for its source and target respectively. 
Arrows in quivers are composed as functions, that is if $ab$ is a path then $\source(a)=\target(b)$. If $Q$ is a quiver, we write $Q_0$ for its set of vertices and $Q_1$ for its set of arrows. 
The \emph{path algebra} $kQ$ is the algebra generated by all paths (including those of length zero) of $Q$, with multiplication induced by composition of paths. We make the somewhat non-standard choice of identifying the vertex $i\in Q_0$ with the stationary path at $i$, in order to avoid clogging the notation. Thus $i\in Q_0$ is an idempotent in $kQ$.

\subsection{Quivers with potentials}\label{subsec:qp}
For a quiver $Q$, we denote by $\widehat{k Q}$ the completion of $kQ$ with respect to the arrow ideal. Thus $\widehat{k Q}$ has a topological base of paths in $Q$, or equivalently its elements are infinite linear combinations of paths in $Q$. 
The space of \emph{potentials} $P_Q$ is the vector space\[
P_Q= \widehat{kQ}/ \overline{[\widehat{kQ}, \widehat{kQ}]},
\]
where $\left[-,-\right]$ denotes the commutator and $\overline{\phantom{xx}}$ the closure. Thus a potential is an infinite linear combination of cyclic paths up to cyclic permutation.
A \emph{quiver with potential} is a pair~$(Q,W)$, where~$Q$ is a quiver and~$W$ is a potential on~$Q$.

We write $\widehat{k Q}_{cyc}$ for the subalgebra of $\widehat{k Q}$ of (possibly infinite) linear combinations of cyclic paths. 
We define a map $\sh: \widehat{kQ}\to \widehat{kQ}$ (see \cite{HI11b}) by
$\sh( a_n \cdots  a_1)= \sum_{i=1}^{n} a_i \cdots  a_1 a_n \cdots  a_{i+1}.$ This induces a map $\sh: 
\widehat{kQ}_{cyc}\to \widehat{kQ}$ which in turn induces a map $\sh: P_Q\to \widehat{kQ}$.

For $ a\in Q_1$, we define a map $\delta_ a:\widehat{kQ}\to \widehat{kQ}$ by 
$$\delta_ a(p) = 
\begin{cases}
q, \text{ if } p =  a q;\\
0, \text{ otherwise.}
\end{cases} $$
Define the partial derivative $\partial_ a : P_Q\to \widehat{kQ}$ by $\partial_ a (W) = \delta_ a(\sh W)$.  The \emph{Jacobian algebra} of a quiver with potential~$(Q,W)$ is defined by
\[
 J(Q,W) = \widehat{kQ}/\overline{(\partial_a W \ | \ a\in Q_1)}.
\]

We recall from \cite{AP17} the construction of the \emph{Ginzburg dg algebra}~$\Gamma_{Q,W}$ associated to a quiver with potential~$(Q,W)$.  First define a graded quiver~$\overline Q$ whose vertices are the same as those of~$Q$, and with arrows as follows:
\begin{itemize}
 \item for each arrow~$a:i\to j$ of~$Q$, there is an arrow~$a:i\to j$ of degree~$0$ in~$\overline Q$;
 \item for each arrow~$a:i\to j$ of~$Q$, there is an arrow~$ a^*:j\to i$ of degree~$-1$ in~$\overline Q$;
 \item for each vertex~$i$ of~$Q$, there is an arrow~$t_i:i\to i$ of degree~$-2$.
\end{itemize}
As a graded algebra, the Ginzburg dg algebra~$\Gamma_{Q,W}$ is defined to be~$\Gamma_{Q,W} = \bigoplus_{m\leq 0} \widehat{k\overline Q}_m$, where~$\widehat{k\overline Q}_m$ is the space of (possibly infinite) linear combinations of paths of degree~$m$ in~$\overline Q$.  The differential is defined on the arrows by
\begin{itemize}
 \item for each arrow~$a:i\to j$ of~$Q$, $d(a) = 0$ and $d( a^*) = \partial_a W$;
 \item for each vertex~$i$ of~$Q$, $d(t_i) = i \big(\sum_{a \in Q_1} [a, a^*]\big) i$,
\end{itemize}
and then extended to all of $\Gamma_{Q, W}$ by the Leibniz rule.

The Ginzburg dg algebra is a ``dg enrichment" of the Jacobian algebra: by construction one sees directly that $H^0(\Gamma_{Q, W})\cong J(Q, W)$.

\subsection{Skew group algebras}\label{subsec:sga}
Let $G$ be a finite abelian group acting on an algebra $A$ by automorphisms. Let us also assume that $|G|\neq 0$ in $k$. We will study the \emph{skew group algebra} $AG$, which is the vector space $A\otimes_k kG$ equipped with the multiplication induced linearly by 
\[
(a\otimes g)(b\otimes h) = ag(b)\otimes gh.
\]
We denote by $ G^\vee$ the group of irreducible characters of $G$.
The group algebra $kG$ is basic semisimple of dimension $|G|$, with a basis given by $\{e_\rho, \, \rho \in  G^\vee\}$, where
\[
e_\rho=\frac{1}{|G|}\sum_{g\in G}\rho(g)g
\]
is an idempotent.
If $H\leq G$, we denote by $e_{\rho |_H}$ the idempotent
\[
e_{\rho|_H}=\frac{1}{|H|}\sum_{g\in H}\rho(g)g
\]
of $kG$. We remark that $e_\rho e_{\rho|_H}= e_\rho$ in $kG$.

If $G$ acts on a finite set $I$, we write $G(i)$ for the orbit of an element $i\in I $ and $G_i$ for the stabilizer of $i$. Since $G$ is abelian, $G_i$ is the stabilizer of $j$ for every $j\in G(i)$. We write $G_{ij} = G_i\cap G_j$.

Let $(Q,W)$ be a quiver with potential, and assume that $G$ acts on $kQ$ by permuting the vertices and stabilizing the arrowspan. Assume moreover that $g(W)=W$ for every $g\in G$.
Then we extend the action of $G$ to an action by dg automorphisms on $\Gamma_{Q, W}$ as follows: $g( a^*) = \sum\lambda_i  b^*_i$, where $g( a) = \sum\lambda_i  b_i,  b_i\in Q_1$. For degree -2, we set $g(t_i) = t_{g(i)}$. We then define the \emph{skew group dg algebra} $\Gamma_{Q, W}G$ as follows: 
\begin{itemize}
 \item as a graded vector space, it is equal to the tensor product~$\Gamma_{Q,W} \otimes_k kG$, where~$kG$ is concentrated in degree~$0$;
 \item multiplication is $k$-linear and defined by~$(x\otimes g)(y\otimes h) = xg(y)\otimes gh$;
 \item the differential is $k$-linear and defined by~$d(x\otimes g) = dx\otimes g$.
\end{itemize}

It follows from \cite[Proposition 2.2 and Corollary 2.3]{AP17} that~$\Gamma_{Q,W}G$ is a dg algebra, and that its cohomology in degree zero is isomorphic to the skew group algebra~$(H^0\Gamma_{Q,W})G$.

\section{Skew group algebras of hereditary algebras by abelian groups}\label{sec:sga}
In Theorem~\ref{thm:main} we are going to prove an isomorphism between the basic version of the skew group algebra of a Ginzburg dg algebra and a certain new Ginzburg dg algebra. In this section we focus on the case where the potential is zero, that is we start from a hereditary algebra (when the potential is zero, the dg structure is trivial). Later we will extend this construction to the case when the potential is nonzero.

For the rest of this section, let then $G$ be a finite abelian group and let $Q$ be a quiver, and assume that $G$ acts linearly on $k Q$ by permuting the vertices and stabilizing the span of the arrows. Our goal in this section is to describe the quiver $Q_G$ of a basic version of the skew group algebra $(k Q)G$, as well as a Morita idempotent $\bar e\in (k Q)G$ such that $k Q_G \cong \bar e(k Q)G\bar e$. Then in the next section we will explicitly construct an isomorphism $\phi :k Q_G \to \bar e(k Q)G\bar e$.

The quiver $Q_G$ has been constructed in \cite{Dem10} for any finite group, but in the abelian case Demonet's description can be simplified. We will start with a lemma which is purely about representations of abelian groups.

Let $V=\bigoplus_{i=1}^n V_i$ be a finite-dimensional vector space. Let $G$ act transitively on $N = \{1, \ldots, n\}$ and also linearly on $V$ such that $g(V_i)\subseteq V_{g(i)}$.

Since $G$ acts by automorphisms, it follows that $V_i\cong V_{g(i)}$ and $g(V_i)=V_{g(i)}$.
Let $S\leq G$ be the stabilizer of $1$ (thus $S$ is the stabilizer of $i$ for all $i$ since $G$ is abelian). Then $S$ acts on $V_i$ for all $i$.

By a \emph{generalized permutation matrix} we mean a matrix with exactly one nonzero entry in every row and in every column.

\begin{lemma}\label{lem:perm}
	There is a basis of $V$ such that all the elements of $G$ act in this basis by generalized permutation matrices.
\end{lemma}

\begin{proof}
	Write $V_1 = \bigoplus_{t=1}^l \chi_t$ where each $\chi_t$ is an irreducible character of $S$.
	Choose a basis $\{v_1, \ldots, v_l\}$ of $V_1$ with the property that
	$s(v_t)= \chi_t(s)v_t$.
	
	Let $R$ be a set of representatives of $G/S$ in $G$.
	Define $\mathcal B=\{g(v_t)\, |\, t=1, \ldots, l, \, g\in R\}$.
	Each set $\{g(v_t)\, |\, t=1, \ldots, l\}$ is a basis of $V_{g(1)}$, so $\mathcal B$ is a basis of $\bigoplus_{g\in R}V_{g(1)}$. On the other hand, since the action is transitive on $N$, we have that $\bigoplus_{g\in R}V_{g(1)}= V$.
	
	Now let $h\in G$ and $g(v_t)\in \mathcal B$. By definition, $hg = g's$ for some $g'\in R$ and $s\in S$. But then $h(g(v_t)) = \chi_t(s)g'(v_t)$ and $g'(v_t)\in \mathcal B$, which proves the claim.
\end{proof}

\begin{nota}\label{not:itilde}
From now on we fix $\tilde I$ to be a set of representatives of $Q_0$ under the action of $G$.
This choice affects the rest of the paper, but it is inevitable. It corresponds to choosing an idempotent subalgebra of $(kQ)G$ which is Morita equivalent to $(kQ)G$.
 Vertices of $Q$ which are known to belong to $\tilde I$ will be denoted by, for instance, $i_\circ$ and $j_\circ$.
\end{nota}

\begin{nota}\label{not:M}
	We denote by $M_{ij}$ the space generated by the arrows from vertex $i$ to vertex $j$.
	Write $M\subseteq k Q$ for the space generated by all the arrows of $Q$. We write $MG\subseteq (kQ)G$ for the space generated by elements of the form $m\otimes g$ with $m\in M$ and $g\in G$.
\end{nota}

\begin{rmk}\label{rmk:wlog}
	Let $i,j\in Q_0$. Let $N$ be the orbit of $(i,j)\in Q_0\times Q_0$ under the diagonal action of $G$.
	Let $V$ be the $G$-orbit of $M_{ij}$. Then by Lemma~\ref{lem:perm} we can choose a suitable basis of $V$ such that $G$ maps arrows in $V$ to multiples of arrows. By repeating this for every orbit in $Q_0\times Q_0$, we can assume that $G$ maps all arrows in $Q$ to multiples of arrows (as opposed to our initial, a priori weaker, assumption that $G$ preserves the arrowspans). We will without loss of generality make this simplifying assumption in the rest of the article.
\end{rmk}

\begin{nota}\label{not:chi}
By the construction of Lemma~\ref{lem:perm}, to every arrow $a:i\to j$ we can associate a character $\chi_a$ of $G_{ij}$ such that for every $g\in G_{ij}$ we have $g(a) = \chi_a(g)a$.
\end{nota}
\begin{nota}\label{not:q0}
We can now describe the vertices of $Q_G$ (see \cite{Dem10}): they are given by 
$$(Q_G)_0 = \left\{(i_\circ,\rho) \,|\, i_\circ \in \tilde I,\, \rho\in  G^\vee_{i_\circ}\right\}.$$
The idempotent of $(kQ)G$ corresponding to the vertex $(i_\circ,\rho)$ is $e_{i_\circ\rho}= i_\circ \otimes e_\rho$, where $$e_\rho= \frac{1}{|G_{i_\circ}|}\sum_{g\in G_{i_\circ}}\rho(g)g$$ is an idempotent of the group algebra $kG_{i_\circ}$.
\end{nota}

\begin{nota}\label{not:ebar}
In particular, the idempotent 
\[
\bar e = \sum_{i_\circ \in \tilde I}\sum_{\rho \in G^\vee_{i_\circ}} e_{i_\circ \rho} 
\]
of $(k Q)G$ is defined.
\end{nota}

\begin{lemma}
 The idempotent~$\bar e$ is such that $\bar e(k Q)G\bar e$ is basic and Morita equivalent to~$(k Q)G$.
\end{lemma}
\begin{proof}
 This follows from~\cite[Section 2.3]{RR85} and~\cite[Theorem 1]{Dem10}.
\end{proof}

\begin{rmk}
	The construction of $Q_G$ does depend on the choice of $\tilde I$. Different choices will result in isomorphic quivers, but in different Morita idempotents $\bar e\in (kQ)G$.
\end{rmk}

Let us now describe the set $(Q_G)_1$. Following \cite{Dem10}, we need to fix some more notation.
\begin{nota}\label{not:kappa}
 For each $i \in Q_0$, choose an element $\kappa_i\in G$ such that $\kappa_i (i)\in \tilde{I}$. We fix $\kappa_{i_\circ}= 1$ for each $i_\circ\in \tilde{I}$.
\end{nota}
\begin{nota}\label{not:distinguished}
 For each $i_\circ, j_\circ\in \tilde I$, choose a set $R_{i_\circ j_\circ}$ of representatives of $G(i_\circ)$ under the action of $G_{j_\circ}$.
Note that $\{(i,j_\circ) \,|\, i\in R_{i_\circ j_\circ}\}$ is then a set of representatives of $G(i_\circ)\times G(j_\circ)$ under the diagonal action of $G$.
Let us define, for each $i_\circ, j_\circ \in \tilde I$,  $$D(i_\circ,j_\circ) = \left\{a: i \to j_\circ\in Q_1,\,  i\in R_{i_\circ j_\circ}
\right\}.$$
We will call $D= \bigcup_{i_\circ, j_\circ \in \tilde{I}}D(i_\circ, j_\circ)$ the set of \emph{distinguished arrows} of $Q$.
\end{nota}

\begin{lemma}\label{lem:q1}
	The set of arrows in $Q_G$ from $(i_\circ,\rho)$ to $(j_\circ,\sigma)$ is in bijection with the set 
	$$\left\{a\in D(i_\circ, j_\circ) \, \big\rvert \, \rho|_{G_{i_\circ j_\circ}} = \sigma|_{G_{i_\circ j_\circ}}\chi_a\right\}.$$
\end{lemma}

\begin{proof}
	By \cite{Dem10}, the set of arrows in $Q_G$ from $(i_\circ,\rho)$ to $(j_\circ,\sigma)$ is in bijection with a basis of
	$$\bigoplus_{i\in R_{i_\circ j_\circ}} \on{Hom}_{G_{i_\circ j_\circ}}\left(\rho|_{G_{i_\circ j_\circ}}, \sigma|_{G_{i_\circ j_\circ}} \otimes_k M_{ij_\circ}\right) =
	\on{Hom}_{G_{i_\circ j_\circ}}\left(\rho|_{G_{i_\circ j_\circ}}, \sigma|_{G_{i_\circ j_\circ}} \otimes_k \bigoplus_{i\in R_{i_\circ j_\circ}} M_{ij_\circ}\right).$$
	Now observe that $$\bigoplus_{i\in R_{i_\circ j_\circ}} M_{ij_\circ} \cong \bigoplus_{a\in D(i_\circ, j_\circ)} \chi_a$$ as $G_{i_\circ j_\circ}$-modules, so that only the arrows $a\in D(i_\circ,j_\circ)$ such that $\rho|_{G_{i_\circ j_\circ}} = \sigma|_{G_{i_\circ j_\circ}}\chi_a$ contribute to the dimension.
\end{proof}

\begin{nota}\label{not:atilde}
	We will denote by $\tilde{a}_{\rho \sigma }$ the arrow (if it exists) from $(i_\circ, \rho)$ to $(j_\circ, \sigma)$ of $Q_G$ corresponding to $a \in D(i_\circ, j_\circ)$.
	
\end{nota}

\section{Explicit isomorphism}\label{sec:iso}
In this section we will expand on Demonet's result, and write an explicit algebra isomorphism $\phi: k Q_G \to \bar e (kQ)G\bar e$.
Let us begin by defining $\phi$ on vertices.
\begin{defin}\label{def:phivertices}
We set 
\[
\phi((i_\circ, \rho))= e_{i_\circ \rho}
\]
for all $i_\circ \in \tilde I$ and $\rho\in  G^\vee_{i_\circ}$. Observe that this is well defined since $e_{i_\circ \rho}$ is a summand of $\bar e$.
\end{defin}

Defining $\phi$ on arrows is the same as choosing a basis of each space
$e_{j_\circ \sigma}MGe_{i_\circ \rho}$.
In order to do this, we make the following definition.
\begin{defin}\label{def:iota}
	We define a function $\iota:kQ \to (kQ)G$ by setting
	\[\iota(p)= \left(1\otimes \kappa_{\target(p)}\right)(p\otimes 1)\left(1\otimes \kappa_{\source(p)}^{-1}\right)
	\]
	for every path $p$ in $Q$, and extending it linearly to $kQ$. In the same way this defines a continuous function  $\iota: \widehat{kQ} \to (\widehat{kQ})G$.
\end{defin}

Remark that $\iota$ is not an algebra morphism, but it has two properties:
\begin{lemma}\label{lem:iota_prop}
	\begin{enumerate}
		\item If $p, q$ are paths in $Q$, and $pq\neq 0$, then $\iota(pq)= \iota(p)\iota(q)$. 
		\item $\iota(kQ)\subseteq \bar e (kQ)G\bar e$. 
	\end{enumerate}
\end{lemma}
\begin{proof}
	\begin{enumerate}
		\item Since $\source(p) = \target(q)$, we have
		\begin{align*}
		\iota(p) \iota(q) &= (1 \otimes \kappa_{\target(p)}) (p \otimes 1) \left(1 \otimes \kappa_{\source(p)}^{-1} \kappa_{\target(q)}\right) (q \otimes 1) \left(1 \otimes \kappa_{\source(q)}^{-1}\right) \\
		&= (1 \otimes \kappa_{\target(p)}) (pq \otimes 1) \left(1 \otimes \kappa_{\source(q)}^{-1}\right) \\
		&= \iota(pq).
		\end{align*}
		\item It is enough to show that $\bar e \iota(p) \bar e = \iota(p)$ for all paths $p$ in $Q$. We first observe that $\bar e$ can be written as $\bar e = \sum_{i_\circ \in \tilde I} i_\circ \otimes 1$. Hence we have
		\begin{align*}
			\bar e \iota(p) \bar e & = \sum_{i_\circ, j_\circ \in \tilde I} (i_\circ \otimes 1) (1 \otimes \kappa_{\target(p)}) (p \otimes 1) \left(1 \otimes \kappa_{\source(p)}^{-1}\right) (j_\circ \otimes 1) \\
			& = \sum_{i_\circ, j_\circ \in \tilde I} (i_\circ \kappa_{\target(p)}(p) \otimes \kappa_{\target(p)}) \left(\kappa_{\source(p)}^{-1}(j_\circ) \otimes \kappa_{\source(p)}^{-1}\right) \\
			& = \sum_{i_\circ, j_\circ \in \tilde I} i_\circ \kappa_{\target(p)}(p) \kappa_{\target(p)} \kappa_{\source(p)}^{-1}(j_\circ) \otimes \kappa_{\target(p)} \kappa_{\source(p)}^{-1} \\
			& = \sum_{i_\circ, j_\circ \in \tilde I} \kappa_{\target(p)}\left(\kappa_{\target(p)}^{-1} (i_\circ) p \kappa_{\source(p)}^{-1}(j_\circ)\right) \otimes \kappa_{\target(p)} \kappa_{\source(p)}^{-1}.
		\end{align*}
		We may note that $\kappa_{\target(p)}^{-1} (i_\circ) p \kappa_{\source(p)}^{-1}(j_\circ)$ is zero unless $\kappa_{\target(p)}^{-1} (i_\circ) = \target(p)$ and $\kappa_{\source(p)}^{-1}(j_\circ) = \source(p)$, in which case 
		$$
		\kappa_{\target(p)}\left(\kappa_{\target(p)}^{-1} (i_\circ) p \kappa_{\source(p)}^{-1}(j_\circ)\right) \otimes \kappa_{\target(p)} \kappa_{\source(p)}^{-1} = \iota(p),$$
		hence the result follows. \qedhere
	\end{enumerate}
\end{proof}

\begin{rmk}
	The choice of $\tilde I$ (or equivalently of $\bar e$) affects the definition of $\iota$ via the choice of the $\kappa_i$.
\end{rmk}

\begin{rmk}
	In general there exists no algebra morphism $kQ\to k Q_G$, but it turns out that the weaker properties of $\iota$ are sufficient for our purposes.
\end{rmk}
\begin{lemma}\label{lem:arrow_basis}
	The set $$
	\left\{(1\otimes e_\sigma)\iota( a)(1\otimes e_\rho)\,|\,  a\in D(i_\circ, j_\circ),\, \chi_ a = \rho|_{G_{i_\circ j_\circ}}\sigma|_{G_{i_\circ j_\circ}}^{-1}
	\right\}$$
	is a basis of $e_{j_\circ \sigma}MGe_{i_\circ \rho}$.
\end{lemma}

\begin{proof}
	By \cite{Dem10}, there is an isomorphism of algebras $kQ_G\cong \bar e (KQ)G\bar e$ mapping $(i_\circ, \rho)$ to $e_{i_\circ\rho}$ and $(j_\circ, \sigma)$ to $e_{j_\circ \sigma}$.
	In particular, arrows in $Q_G$ from $(i_\circ, \rho)$ to $(j_\circ, \sigma)$ are in bijection with a basis of $e_{j_\circ \sigma}MGe_{i_\circ \rho}$. On the other hand, such arrows are in bijection with the set of the statement by Lemma~\ref{lem:q1}. Therefore it is enough to show that this set generates $V= e_{j_\circ \sigma}MGe_{i_\circ \rho}$ as a vector space.
	
	The set $\left\{ e_{j_\circ \sigma}( a\otimes e_\chi)e_{i_\circ \rho}\, | \,  a \in Q_1, \, \chi \in G^\vee\right\}$ generates $V$ by definition. 
	Now $(j_\circ \otimes e_\sigma)(a\otimes e_\chi)(i_\circ \otimes e_\rho) = 0$ unless $j_\circ a\neq 0$, so we can assume that $\target(a) = j_\circ$. 
	Then 
	\begin{align*}
	(j_\circ \otimes e_\sigma)(a\otimes e_\chi)(i_\circ \otimes e_\rho) = 
	\frac{1}{|G|}\sum_{g\in G} (1\otimes e_\sigma)(a g(i_\circ)\otimes g \chi(g)e_\rho),
	\end{align*}
	which is 0 unless $\source(a)= i$ for some $i\in G(i_\circ)$. 
	In this case, recalling that $i_\circ = \kappa_i(i)$, the above expression equals
	\begin{align*}
	& \frac{1}{|G|}\sum_{g \in G, g(i_\circ)=i} (1\otimes e_\sigma)(a \otimes g \chi(g)e_\rho) \\
	& = \frac{1}{|G|}\sum_{g \kappa_i \in G_{i_\circ}} (1\otimes e_\sigma)(a \otimes g \chi(g)e_\rho) \\
	& = \frac{1}{|G|}\sum_{g \in G_{i_\circ}} (1\otimes e_\sigma)\left(a\otimes \kappa_i^{-1}
	g\chi(g)\chi(\kappa_i^{-1})e_\rho\right) \\
	& = \chi(\kappa_i^{-1})\frac{|G_{i_\circ}|}{|G|} 
	(1\otimes e_\sigma)\iota(a)(1\otimes e_{\rho}e_{\chi|_{G_{i_\circ}}}) .
	\end{align*}
	The latter is either zero (if $\rho \neq \chi|_{G_{i_\circ}}$) or a scalar multiple of $(1\otimes e_\sigma)\iota( a)(1\otimes e_\rho)$. 
	Since there exists at least one $\chi \in G^\vee$ such that $\rho = \chi|_{G_{i_\circ}}$, we conclude that the set 
	$$
	\left\{(1\otimes e_\sigma)\iota( a)(1\otimes e_\rho)\,|\,  a: i \to j_\circ\in Q_1, \, i \in G(i_\circ)
	\right\}$$
	generates $V$.
	
	Now observe that the action of $G_{j_\circ}$ stabilizes the above set, up to scalars: if $g\in G_{j_\circ}$ and $g(a) =\xi b$ for some arrow $b$, we have
	\begin{align*}
	(1\otimes e_\sigma)\iota(g(a))(1\otimes e_\rho)&= 
	\xi (1\otimes e_\sigma)(b\otimes 1)(1\otimes \kappa_{\source(b)}e_\rho) \\
	&= \xi (1\otimes e_\sigma)\iota(b)(1\otimes e_\rho).
	\end{align*}
	So it is enough, in order to generate $V$, to take a set of representatives of arrows modulo the action of $G_{j_\circ}$, i.e.~we can assume $\source(a)\in R_{i_\circ j_\circ}$. 
	
	Finally, notice that 
	\begin{align*}
	(1\otimes e_\sigma)\iota(a)(1\otimes e_\rho)&=
	(1\otimes e_\sigma e_{\sigma|_{G_{i_\circ j_\circ}}})
	(a\otimes 1)(1\otimes \kappa_{\source(a)})(1\otimes e_{\rho|_{G_{i_\circ j_\circ}}}e_\rho)  \\
	&= \frac{1}{|G_{i_\circ j_\circ}|}\sum_{g\in G_{i_\circ j_\circ}} (1\otimes e_\sigma)(g(a)\otimes 
	g\sigma(g)e_{\rho|_{G_{i_\circ j_\circ}}})(1\otimes e_\rho)(1\otimes \kappa_{\source(a)}) \\
	&= \frac{1}{|G_{i_\circ j_\circ}|}\sum_{g\in G_{i_\circ j_\circ}} (1\otimes e_\sigma)(a\otimes 
	g\chi_a(g)\sigma(g)e_{\rho|_{G_{i_\circ j_\circ}}})(1\otimes e_\rho)(1\otimes \kappa_{\source(a)}) \\
	&= (1\otimes e_\sigma)(a\otimes e_{\chi_a \sigma|_{G_{i_\circ j_\circ}}}e_{\rho|_{G_{i_\circ j_\circ}}})
	(1\otimes e_\rho)(1\otimes \kappa_{\source(a)}).
	\end{align*}
	Thus $a$ must be such that $\chi_a = \rho|_{G_{i_\circ j_\circ}}\sigma|_{G_{i_\circ j_\circ}}^{-1}$ as claimed.
\end{proof}
\begin{defin}\label{def:phiarrows}
	Recall that for every arrow $\tilde a_{\rho\sigma }$ of $Q_G$, there exists a unique corresponding distinguished arrow $a\in Q$. So we can set
	\[
	\phi(\tilde{a}_{\rho\sigma}) = (1\otimes e_\sigma)\iota(a)(1\otimes e_\rho).
	\]
	Thus $\phi$ gives an isomorphism of vector spaces $k (Q_G)_1\cong \bar eMG\bar e$ (it is a bijection between bases), compatible with the definition of $\phi$ on vertices.
	We extend it uniquely to a morphism of algebras $\phi:k Q_G \to \bar e(k Q)G\bar e$, which is then also an isomorphism.
	As we did for $\iota$, we can also define in the same way a continuous algebra isomorphism $\phi:\widehat{k Q_G}\to \bar e(\widehat{k Q})G\bar e$.
\end{defin}

There is also another ``dual'' basis of $e_{j_\circ \sigma}MGe_{i_\circ \rho}$ which we will use:
\begin{lemma}
	\label{lem:trifoglio}
	The set $$
	\left\{(1\otimes e_\sigma)\iota( b)(1\otimes e_\rho)\,| \,  b: i_\circ \to j\in Q_1, \, j\in R_{j_\circ i_\circ}, \, \chi_b = \rho|_{G_{i_\circ j_\circ}}\sigma|_{G_{i_\circ j_\circ}}^{-1}
	\right\}$$
	is a basis of $e_{j_\circ \sigma}MGe_{i_\circ \rho}$.
\end{lemma}

\begin{proof}
	First observe that by the well-known orbit-counting lemma we have
	\[
	|R_{i_\circ j_\circ}| = \frac{|G||G_{i_\circ j_\circ}|}{|G_{i_\circ}||G_{j_\circ}|}	= |R_{j_\circ i_\circ}|.
	\]
	Moreover, by a similar computation as in the end of the proof of Lemma~\ref{lem:arrow_basis}, we can write
	\begin{align*}
	(1\otimes e_\sigma)\iota(b)(1\otimes e_\rho)=
	(1\otimes e_\sigma)(1\otimes \kappa_{\target(b)})(b\otimes e_{\chi_b \sigma|_{G_{i_\circ j_\circ}}}e_{\rho|_{G_{i_\circ j_\circ}}})
	(1\otimes e_\rho),
	\end{align*}
	so that $(1\otimes e_\sigma)\iota(b)(1\otimes e_\rho)$ is 0 unless $\chi_b = \rho|_{G_{i_\circ j_\circ}}\sigma|_{G_{i_\circ j_\circ}}^{-1}$.
	
	Taken together, these two observations imply that the set in the statement has the correct cardinality. We shall show that every element of the form 
	$$(1\otimes e_\sigma)\iota( a)(1\otimes e_\rho)$$
	such that $a\in D(i_\circ, j_\circ)$ is a scalar multiple of an element in this set, and this implies the statement by Lemma~\ref{lem:arrow_basis}.
	
	First we remark that any $a:i\to j_\circ$ can be written as 
	$a = \kappa_i^{-1} h (b)$, where:
	\begin{itemize}
		\item $b: i_\circ \to j$ is an arrow with $ j\in R_{j_\circ i_\circ}$,
		\item $h\in G_{i_\circ}$.
	\end{itemize}
    Moreover, as $\kappa_i^{-1} h$ maps $j$ to $j_\circ$, we must have 
    $\kappa_i^{-1} h = l\kappa_{j}$ for some $l\in G_{j_\circ}$.
    We obtain
    \begin{align*}
    	(1\otimes e_\sigma)\iota( a)(1\otimes e_\rho) &= 
    	(1\otimes e_\sigma)( a\otimes 1)\left(1\otimes \kappa_i^{-1}\right)(1\otimes e_\rho)\\
    	& =(1\otimes e_\sigma)\left(1\otimes \kappa_i^{-1}\right)( h(b)\otimes 1)(1\otimes e_\rho)\\
    	&= (1\otimes e_\sigma)\left(1\otimes \kappa_i^{-1}h\right)( b\otimes 1)\left(1\otimes e_\rho h^{-1}\right)\\
    	&= (1\otimes e_\sigma l)(1\otimes \kappa_j)( b\otimes 1)\left(1\otimes e_\rho h^{-1}\right) \\
    	&= \sigma(l^{-1})\rho(h) (1\otimes e_\sigma)\iota( b)(1\otimes e_\rho).
    \end{align*}
    This concludes the proof.
\end{proof}

The main advantage of defining $\phi$ with the help of $\iota$ is that we can carry out explicit computations relatively easily. In particular, we can express the image via $\iota$ of a path in $Q$ as an explicit linear combination of paths in $Q_G$. We show such formulas in the next results.

\begin{lemma}\label{lem:iota(arrow)}
	Let $ a$ be a distinguished arrow in $Q$. Let $g\in G$ be such that $b = g^{-1}( a)$ is an arrow. Then 
	\begin{align*}
	\phi^{-1}\iota(b) = \sum_{\rho\in G_{\source(b)}^\vee, \, \sigma\in G_{\target(b)}^\vee} \sigma\left(g\kappa_{\target(b)}^{-1}\right)\rho\left(g^{-1}\kappa_{\source(a)}^{-1}\kappa_{\source(b)}\right)\tilde  a_{\rho \sigma}.
	\end{align*}
\end{lemma}

\begin{proof}
	Observe first that $g\kappa_{\target(b)}^{-1}\in G_{\target(b)}$ and that $g^{-1}\kappa_{\source(a)}^{-1}\kappa_{\source(b)}\in G_{\source(b)}$. We have
	\begin{align*}
	\iota(b) &= 
	\iota(g^{-1}( a))  \\
	&=(1\otimes \kappa_{\target(b)})(g^{-1}( a)\otimes 1)\left(1\otimes \kappa_{\source(b)}^{-1}\right)  \\
	&=\left(1\otimes \kappa_{\target(b)}g^{-1}\right)( a\otimes 1)\left(1\otimes \kappa_{\source(a)}^{-1}\right)\left(1\otimes \kappa_{\source(b)}^{-1}\kappa_{\source(a)} g \right)\\
	&= \sum_{\rho\in G_{\source(b)}^\vee, \, \sigma\in G_{\target(b)}^\vee}\sigma\left(g\kappa_{\target(b)}^{-1}\right)\rho\left(g^{-1}\kappa_{\source(a)}^{-1}\kappa_{\source(b)}\right)(1\otimes e_\sigma)\iota( a)(1\otimes e_\rho),
	\end{align*}
	and the claim follows by applying $\phi^{-1}$.
\end{proof}

To use this formula on potentials, it is convenient to compute it for cycles. In fact, in the proof of Theorem~\ref{thm:main} we will only use the formula of Lemma~\ref{lem:iota(arrow)}.
However, one could say that the main new tool this article introduces is given by the formulas of Proposition~\ref{prop:iota(path)}. While these are not needed to prove our result, they are what one uses in practice to compute examples (as we illustrate in Section~\ref{sec:example}). 
Before proving them we shall introduce some additional notation.

\begin{nota}\label{not:cyc}
	We set $\cyc:kQ\to kQ$ to be the linear map defined on a path $p$ as $\cyc(p) = \sum_{i\in Q_0} ipi$. As usual, we extend it to a continuous map $\cyc: \widehat{k Q}\to \widehat{k Q}$.
\end{nota}

Now we use the previous lemma to write the formula for $\iota$ of an arbitrary path, and that of $\cyc \iota$ for a cycle. 

\begin{prop}\label{prop:iota(path)}
	Let $b_n\cdots b_1$ be a nonzero path in $Q$ (with the $b_i$ arrows), and choose for each $i= 1, \dots, n$ an element $g_i\in G$ such that $g_i(b_i)= a_i \in D$.
	We call $T= \left(\prod_{i=1}^{n} G_{\source(b_i)}^\vee\right) \times G_{\target(b_n)}^\vee$ and use the notation $\underline \sigma = (\sigma_1, \dots, \sigma_{n+1})$ to denote elements in $T$.
	Then
	\begin{align*}
	\phi^{-1}\iota(b_n\cdots b_1) &= 
	\sum_{\underline \sigma \in T}
	\sigma_{n+1}\left(g_n\kappa_{\target(b_n)}^{-1}\right)\sigma_1\left(g_1^{-1}\kappa_{\source(a_1)}^{-1}\kappa_{\source(b_1)}\right)\cdot \\
	&\cdot\prod_{i= 2}^{n} \sigma_i\left(g_i^{-1}g_{i-1}\kappa_{\source(a_i)}^{-1}\right)
	\tilde a_{n,\sigma_n \sigma_{n+1}}\tilde a_{n-1, \sigma_{n-1}\sigma_n}\cdots
	\tilde a_{1, \sigma_2\sigma_1}.
	\end{align*}
	In particular, if $b_n\cdots b_1$ is a cycle, we have the nicer formula
	\begin{align*}
	\cyc(\phi^{-1}\iota(b_n\cdots b_1)) = 
	\sum_{\underline \sigma \in \prod_{i=1}^{n} G_{\source(b_i)}^\vee}
	\prod_{i= 1}^{n} \sigma_i\left(g_i^{-1}g_{i-1}\kappa_{\source(a_i)}^{-1}\right)
	\tilde a_{n,\sigma_n \sigma_{1}}\tilde a_{n-1, \sigma_{n-1}\sigma_n}\cdots
	\tilde a_{1, \sigma_2\sigma_1},
	\end{align*}
	with the convention $g_0= g_n$.
\end{prop}

\begin{proof}
	Using Lemma~\ref{lem:iota_prop}(1) and Lemma~\ref{lem:iota(arrow)} we have
	\begin{align*}
	\iota(b_n \cdots b_1) & = \iota(b_n) \cdots \iota(b_1) \\
	& = \iota\left(g_n^{-1}(a_n)\right) \cdots \iota\left(g_1^{-1}(a_1)\right) \\
	& = \sum_{\substack{\sigma_n \in G_{\source(b_n)}^\vee \\ \sigma_n' \in G_{\target(b_n)}^\vee}}
	\sigma_n'\left(g_n \kappa_{\target(b_n)}^{-1}\right) \sigma_n\left(g_n^{-1} \kappa_{\source(a_n)}^{-1} \kappa_{\source(b_n)}\right)
	\tilde{a}_{n, \sigma_n \sigma_n'} \cdot \\
	& \cdot \sum_{\substack{\sigma_{n-1} \in G_{\source(b_{n-1})}^\vee\\ \sigma_{n-1}' \in G_{\target(b_{n-1})}^\vee}}
	\sigma_{n-1}'\left(g_{n-1} \kappa_{\target(b_{n-1})}^{-1}\right) \sigma_{n-1}\left(g_{n-1}^{-1} \kappa_{\source(a_{n-1})}^{-1} \kappa_{\source(b_{n-1})}\right)
	\tilde{a}_{{n-1}, \sigma_{n-1} \sigma_{n-1}'} \cdot \\
	& \cdots \\
	& \cdot \sum_{\substack{\sigma_1 \in G_{\source(b_1)}^\vee \\ \sigma_1' \in G_{\target(b_1)}^\vee}}
	\sigma_1'\left(g_1 \kappa_{\target(b_1)}^{-1}\right) \sigma_1\left(g_1^{-1} \kappa_{\source(a_1)}^{-1} \kappa_{\source(b_1)}\right)
	\tilde{a}_{1, \sigma_1 \sigma_1'}.
	\end{align*}
	Since $\tilde{a}_{i, \sigma_i \sigma_i'} \tilde{a}_{{i-1}, \sigma_{i-1} \sigma_{i-1}'} = 0$ if $\sigma_{i-1}' \neq \sigma_i$, for all $i = 2, \dots, n$, the above formula is reduced to
	$$
	\iota(b_n \cdots b_1) = 
	\sum_{\underline \sigma \in T} \xi_{\underline \sigma}\, 
	\tilde a_{n,\sigma_n \sigma_{n+1}}\tilde a_{n-1, \sigma_{n-1}\sigma_n} \cdots \tilde a_{1, \sigma_2\sigma_1}
	$$
	(note that we renamed $\sigma_n'$ to $\sigma_{n+1}$) where the coefficient $\xi_{\underline \sigma}$ is given by
	\begin{align*}
	\xi_{\underline \sigma} & =
	\sigma_{n+1}\left(g_n \kappa_{\target(b_n)}^{-1}\right) \sigma_n\left(g_n^{-1} \kappa_{\source(a_n)}^{-1} \kappa_{\source(b_n)}\right)
	\sigma_n\left(g_{n-1} \kappa_{\target(b_{n-1})}^{-1}\right) \sigma_{n-1}\left(g_{n-1}^{-1} \kappa_{\source(a_{n-1})}^{-1} \kappa_{\source(b_{n-1})}\right) \cdots \\
	& \cdots \sigma_2\left(g_1 \kappa_{\target(b_1)}^{-1}\right) \sigma_1\left(g_1^{-1} \kappa_{\source(a_1)}^{-1} \kappa_{\source(b_1)}\right) = \\
	& = \sigma_{n+1}\left(g_n\kappa_{\target(b_n)}^{-1}\right)\sigma_1\left(g_1^{-1}\kappa_{\source(a_1)}^{-1}\kappa_{\source(b_1)}\right)
	\prod_{i= 2}^{n} \sigma_i\left(g_i^{-1}g_{i-1}\kappa_{\source(a_i)}^{-1}\right)
	\end{align*}
	and so the first statement follows.
	
	For the second statement, note that, applying $\cyc$ to the formula we just proved for $\iota(b_n \cdots b_1)$, all the terms of the sum where $\sigma_{n+1} \neq \sigma_1$ become zero.
	Thus the claim follows from the fact that $\target(b_n) = \source(b_1)$.
\end{proof}

\section{Main result}\label{sec:mainresult}
In this section we extend the isomorphism of Section~\ref{sec:iso} from the case of a path algebra to the case of the Ginzburg dg algebra of a quiver with potential.
The setting is as follows: let $Q$ be a quiver with an action of a finite abelian group $G$ on the path algebra $k Q$. As before, we assume that the action permutes the vertices and maps arrows to multiples of arrows. 
Recall that, by Remark~\ref{rmk:wlog}, the weaker assumption that $G$ permutes the vertices and fixes the vector space of all arrows would be sufficient. Moreover, let $W$ be a potential on $Q$ such that $g(W)= W$ for every $g\in G$.
We will define a potential $W_G$ on $Q_G$ such that there is an isomorphism $\Phi$ between the Ginzburg dg algebra $\Gamma_{Q_G, W_G}$ and the idempotent subalgebra $\bar e \Gamma_{Q, W}G\bar e$ of the skew group dg algebra of the Ginzburg dg algebra of $(Q, W)$; this idempotent algebra will be Morita equivalent to~$\Gamma_{Q, W}G$. In degree 0, this isomorphism will be the given by $\phi: \widehat{k Q_G} \to \bar e(\widehat{kQ}) G\bar e$, as defined in Definition~\ref{def:phivertices} and Definition~\ref{def:phiarrows}

\begin{rmk}
	The existence of such a potential $W_G$ was proved in \cite{LM18} for $G$ an arbitrary finite group. Our aim in the abelian case is to compute it explicitly as a (possibly infinite) linear combination of cycles of $Q_G$.
\end{rmk}

In order to define the potential $W_G$, we will make use of the function $\iota:\widehat{k Q}\to (\widehat{k Q})G$.

\begin{lemma}
	There exists a unique continuous map $\bar \iota: P_Q\to P_{Q_G}$ such that the diagram
	$$
	\xymatrix{
	\widehat{kQ}_{cyc} \ar[r]^{\phi^{-1}\iota} \ar@{->>}[d] & \widehat{kQ_G} \ar@{->>}[d]\\
	P_Q \ar[r]^{\bar\iota} & P_{Q_G}
}$$
commutes.
\end{lemma}

\begin{proof}
	What needs to be checked is that~$\phi^{-1}\iota\left(\, \overline{[\widehat{kQ}_{cyc}, \widehat{kQ}_{cyc}]} \,\right)$ is contained in~$\overline{[\widehat{kQ_G}, \widehat{kQ_G}]}$.  For this, it is sufficient to prove the inclusion before taking the closure of the subspaces.  Now, if~$p$ and~$q$ are paths in~$\widehat{kQ}_{cyc}$, then~
	\[
	 [p,q] = \begin{cases}
	          pq-qp & \textrm{if $\source(p) = \source(q)$}; \\
	          0 & \textrm{otherwise.}
	         \end{cases}
	\]
	Thus, by Lemma~\ref{lem:iota_prop}, we have that~
	\[
	 \iota([p,q]) = \begin{cases}
	          \iota(p)\iota(q)-\iota(q)\iota(p) = [\iota(p),\iota(q)] & \textrm{if $\source(p) = \source(q)$}; \\
	          0 & \textrm{otherwise.}
	         \end{cases}
	\]
   Therefore,~$\phi^{-1}\iota([p,q])$ lies in~$[\widehat{kQ_G}, \widehat{kQ_G}]$.  By linearity, this proves that~$\phi^{-1}\iota\big({[\widehat{kQ}_{cyc}, \widehat{kQ}_{cyc}]} \big)$ is contained in~${[\widehat{kQ_G}, \widehat{kQ_G}]}$.  Taking the closure of the subspaces, we get the desired result.
\end{proof}

Now we are ready to define the potential $W_G$. 
\begin{defin}\label{def:wg}
	Let $(Q, W)$ be a quiver with potential with an action of a finite abelian group $G$ which permutes the vertices and maps arrows to multiples of arrows. Assume moreover that $g(W)=W$ for all $g\in G$. Then define a potential $W_G$ on the quiver $Q_G$ by $$W_G= \bar\iota(W).$$
\end{defin}

\begin{rmk}\label{rem:lift}
	If $W'$ is a lift of $W$ to $\widehat{kQ}_{cyc}$, we remark that a possible lift of $W_G$ to $\widehat{k Q_G}$ (and in fact to $\widehat{kQ_G}_{cyc}$) is given by $\cyc\phi^{-1}\iota (W')$. 
\end{rmk}

\begin{lemma}\label{lem:s_commutes}
	We have $\sh\phi^{-1}\iota = \cyc\phi^{-1}\iota \sh$ as functions $\widehat{kQ}_{cyc}\to \widehat{kQ_G}_{cyc}$. 
\end{lemma}

\begin{proof}
	First observe that if $A_1, \ldots, A_n$ are sets of scalar multiples of arrows, we have
	\begin{align*}
	\left(\sum_{ a \in A_n} a\right)	\left(\sum_{ a \in A_{n-1}} a\right)\cdots	\left(\sum_{ a \in A_1} a\right) &=
	\sum_{\underline{i}\in (Q_0)^{n+1}} i_n\left(\sum_{ a \in A_n} a\right)i_{n-1}	\left(\sum_{ a \in A_{n-1}} a\right)i_{n-2}\cdots 	i_1\left(\sum_{ a \in A_1} a\right)i_0. 
	\end{align*}
	Now let $c =  a_n\cdots  a_1\in \widehat{kQ}_{cyc}$. 
	We have
	\begin{align*}
	\cyc\phi^{-1}\iota \sh(c) &=
	\cyc\phi^{-1}\iota \sum_{j=1}^n a_j\cdots  a_1 a_n\cdots  a_{j+1} \\
	&= \cyc\sum_{j=1}^n \phi^{-1}\iota( a_j)\cdots \phi^{-1}\iota( a_1)\phi^{-1}\iota( a_n)\cdots \phi^{-1}\iota( a_{j+1}) \\
	&= \sum_{j=1}^n \sum_{\underline{i}\in (Q_0)^{n+1}}i_{j}\phi^{-1}\iota( a_j)i_{j-1}\cdots i_1\phi^{-1}\iota( a_1)i_0\phi^{-1}\iota( a_n)i_{n-1}\cdots i_{j+1}\phi^{-1}\iota( a_{j+1})i_j \\
	&= \sum_{\underline{i}\in (Q_0)^{n+1}}\sh\left( i_n\phi^{-1}\iota( a_n)i_{n-1}\cdots i_1\phi^{-1}\iota( a_1)i_0 \right)\\
	&= \sh\left(\phi^{-1}\iota( a_n)\cdots \phi^{-1}\iota( a_1)
	\right) \\
	&= \sh\phi^{-1}\iota(c).\qedhere
	\end{align*}
\end{proof}

\begin{prop}\label{prop:eGammae}
	The algebra $\bar e \Gamma_{Q, W}G\bar e$ is a dg algebra Morita equivalent to~$\Gamma_{Q,W} G$.
\end{prop}
\begin{proof}
 That~$\bar e \Gamma_{Q, W}G\bar e$ is a dg algebra follows from the fact that for any homogeneous element~$x$ of~$\Gamma_{Q,W}$, we have
 \[
  d(\bar e x \bar e) = d(\bar e)x\bar e + \bar e d(x) \bar e \pm \bar e x d(\bar e) = \bar e d(x) \bar e,
 \]
 since~$d(\bar e) = 0$.
 
 The proof that~$\bar e \Gamma_{Q, W}G\bar e$ and~$\Gamma_{Q,W} G$ are Morita equivalent is very similar to the proof in~\cite{RR85,Dem10} in the non dg case.  We outline the proof for the dg case here.  To lighten the notation, we write~$\Gamma$ instead of~$\Gamma_{Q,W}$.
 
 If we forget the dg structure of~$\Gamma G$, then write~
 \[
  \Gamma G = \bigoplus_{i=1}^r P_i^{\oplus a_i},
 \]
 where the~$P_i$ are pairwise non-isomorphic indecomposable projective (non dg)~$\Gamma G$-modules.  Then the quotient of the endomorphism algebra of~$\Gamma G$ (as a non dg module) by its Jacobson radical is isomorphic to
 \[
  \prod_{i=1}^r \operatorname{Mat}_{a_i \times a_i}(k).
 \]
 One then shows that the image of~$\bar e$ in this quotient is a sum of~$r$ idempotents, one in each copy of~$\operatorname{Mat}_{a_i \times a_i}(k)$.  This shows that this quotient only depends on the action of~$G$ on the vertices of~$Q$, and that~$\bar e \Gamma G \bar e$ is Morita equivalent to~$\Gamma G$.
 
 To prove that this is still true for the dg algebras, it suffices to show that the quotient of the endomorphism algebra of~$\Gamma G$ (as a dg module, this time) by its Jacobson radical is also isomorphic to 
 \[
  \prod_{i=1}^r \operatorname{Mat}_{a_i \times a_i}(k).
 \]
 But we have that
 \begin{eqnarray*}
  \operatorname{End}_{\Gamma G}(\Gamma G) \big/ \operatorname{rad}\big(\operatorname{End}_{\Gamma G}(\Gamma G)\big)  & \cong & H^0(\Gamma G) \big/ \operatorname{rad}(\Gamma G) \\
  & \cong & J(Q,W)G \big/ \operatorname{rad}(J(Q,W)G).
 \end{eqnarray*}
 This last algebra is an associative algebra in degree zero, so~\cite{RR85,Dem10} apply.  Thus it only depends on the action of~$G$ on the vertices of~$Q$, and it is isomorphic to~$\prod_{i=1}^r \operatorname{Mat}_{a_i \times a_i}(k)$.  It is easy to see that the image of~$\bar e$ in this product ring is the same as above.  This finishes the proof. 
\end{proof}

We extend $\iota$ to $\overline Q$ with the same definition.

	\begin{thm}\label{thm:main}
		There is a continuous isomorphism of dgas 
		$\Phi: \Gamma_{Q_G, W_G} \to \bar e \Gamma_{Q, W}G\bar e$ defined by:
		\begin{itemize}
			\item In degree 0, $\Phi = \phi$. 
			\item In degree -1, $\Phi(\tilde  a_{\rho \sigma}^*)
			= \frac{|G|}{|G_{i_\circ j_\circ}|}(1\otimes e_\rho)\iota( a^*)(1\otimes e_\sigma)$, for $ a\in D(i_\circ, j_\circ)$.
			\item In degree -2, $\Phi(t_{(i_\circ, \rho)})
			= \frac{|G|}{|G_{i_\circ}|}(1\otimes e_\rho)\iota(t_{i_\circ})(1\otimes e_\rho)$, 
		\end{itemize}
		extended $k$-linearly, continuously and multiplicatively. 
	\end{thm}

\begin{proof}
	Firstly, $\Phi$ is a map of algebras, since it is by definition $k$-linear, multiplicative, and $\Phi(1) = \phi(1) = 1$. It is also continuous by definition.
	Let us check that $\Phi$ is a vector space isomorphism. 
	In degree 0, this is true since $\Phi = \phi$ and $\phi$ is an isomorphism.
	In degree -1, we need to check that 
	$\left\{
	(1\otimes e_\rho)\iota( a^*)(1\otimes e_\sigma) \, |\,  a\in D
	\right\}
	$
	is a basis of $\left (  \bar e \Gamma_{Q, W}G\bar e  \right )_{-1}$. This follows from Lemma~\ref{lem:trifoglio} applied to the quiver $\overline Q\setminus Q_1$.
	
	In degree -2, it is enough to prove that the vector subspace
	$$V=\left\langle e_{j_\circ\sigma}(t_l\otimes e_\chi)e_{i_\circ\rho} \,|\, l\in Q_0,\, \chi\in G^\vee,\, i_\circ,j_\circ\in\tilde I,\, \rho\in G^\vee_{i_\circ},\, \sigma\in G^\vee_{j_\circ} \right\rangle$$
	of $\bar e \Gamma_{Q, W}G\bar e$ is equal to $\langle \Phi(t_{(i_\circ,\rho)}) \,|\, i_\circ\in\tilde I,\, \rho\in  G^\vee_{i_\circ}\rangle$.
	In the same way as in the proof of Lemma~\ref{lem:arrow_basis}, we can show that $e_{j_\circ\sigma}(t_l\otimes e_\chi)e_{i_\circ\rho}=0$ if $l\neq j_\circ$ or $l\not\in G(i_\circ)$.
	This means that we can consider only such elements for $l=i_\circ=j_\circ$.
	Moreover, again as in Lemma~\ref{lem:arrow_basis}, we can prove that $V$ is generated by the set
	$$\left\{(1\otimes e_\sigma)\iota(t_{i_\circ})(1\otimes e_\rho) \,|\, i_\circ\in\tilde I,\, \rho,\sigma\in G^\vee_{i_\circ} \right\}.$$
	Then the claim follows, since $(1\otimes e_\sigma)\iota(t_{i_\circ})(1\otimes e_\rho) = \iota(t_{i_\circ})(1\otimes e_\sigma e_\rho) = 0$ if $\rho\neq\sigma$.
	
	Let us now check that $\Phi$ commutes with differentials. It is enough to check that $d\Phi(\tilde  a_{\rho \sigma}^*)  = \Phi d(\tilde  a_{\rho \sigma}^*)$ for $ a \in D$, and that $d\Phi(t_{(i_\circ, \rho)})  = \Phi d(t_{(i_\circ, \rho)})$. 
	Let us check the first equality. We have
	\begin{align*}
	d\Phi(\tilde  a_{\rho \sigma}^*) &=
	d\left(\frac{|G|}{|G_{i_\circ j_\circ}|}(1\otimes e_\rho)\iota( a^*)(1\otimes e_\sigma) \right)  \\
	&= \frac{|G|}{|G_{i_\circ j_\circ}|}(1\otimes e_\rho\kappa_{\source(a)})(d( a^*)\otimes 1)(1\otimes e_\sigma) \\
	&= \frac{|G|}{|G_{i_\circ j_\circ}|}(1\otimes e_\rho)\iota(\partial_ a(W))(1\otimes e_\sigma),
	\end{align*}
	using the Leibniz rule. Let us fix a lift of $W$ to $\widehat{kQ}_{cyc}$, which we still call $W$ for simplicity. Let us write $\sh W  = \sum_{c\in \mathcal C} \lambda_c c$, where the $c$ are cycles in $\widehat{kQ}_{cyc}$. 
	Then 
	\begin{align*}
	d\Phi(\tilde  a_{\rho \sigma}^*) &=
	\frac{|G|}{|G_{i_\circ j_\circ}|}(1\otimes e_\rho)\iota(\delta_ a(\sh W))(1\otimes e_\sigma)\\
	&=\frac{|G|}{|G_{i_\circ j_\circ}|}\sum_{c\in \mathcal C, \, c=  a q}(1\otimes e_\rho)\iota(q)(1\otimes e_\sigma).
	\end{align*}
	On the other hand, we have
	\begin{align*}
	\Phi d (\tilde  a_{\rho \sigma}^*) &=
	\Phi\partial_{\tilde  a_{\rho \sigma}}(W_G)  \\
	&= \Phi \delta_{\tilde  a_{\rho \sigma}}\sh \phi^{-1} \iota(W)  \\
	&= \Phi \delta_{\tilde  a_{\rho \sigma}}\cyc \phi^{-1}\iota (\sh W) \\	
	&= \sum_{c\in \mathcal C} \lambda_c \Phi \delta_{\tilde  a_{\rho \sigma}}\cyc \phi^{-1}\iota (c)  \\
	&= \sum_{c\in \mathcal C} \lambda_c \Phi \delta_{\tilde  a_{\rho \sigma}}\cyc (\phi^{-1}\iota(b)\phi^{-1}\iota(p))\\
	&=\sum_{c\in \mathcal C, \, c=  b p} \lambda_c \Phi \delta_{\tilde  a_{\rho \sigma}}\cyc\left ( \left (\sum_{\tau, \omega} \mu_{ b \tau\omega}\tilde{ b'}_{\tau\omega}\right )\phi^{-1}\iota (p) \right) \\
	&=\sum_{c\in \mathcal C, \, c=  b p} \lambda_c \sum_{\tau, \omega} \mu_{ b \tau\omega}\Phi \delta_{\tilde  a_{\rho \sigma}}\tilde{ b'}_{\tau\omega}\phi^{-1}\iota (p) \target(\tilde{ b'}_{\tau\omega}),
	\end{align*}
where $b$ is an arrow, $b'$ is the only distinguished arrow in the orbit of $b$, $\tau$ and $\omega$ run through the characters of the stabilizers of its source and target, and the $\mu_{b\tau \omega}$ are some coefficients.
Note that in the third equality we have applied Lemma~\ref{lem:s_commutes}.

 We can remark that the term in the sum is 0 unless $\tilde{ b'}_{\tau\omega} = \tilde  a_{\rho \sigma}$ by the definition of cyclic derivative, so in fact the above expression equals
	\begin{align*}
	\Phi d (\tilde  a_{\rho \sigma}^*) &=
	\sum_{c\in \mathcal C, \, c=  b p, \,  b \in G( a)} \lambda_c \mu_{ b \rho\sigma}\Phi (i_\circ, \rho)\phi^{-1}\iota(p)(j_\circ, \sigma)\\
	&=\sum_{c\in \mathcal C, \, c=  b p, \,  b \in G( a)} \lambda_c \mu_{ b \rho\sigma}\Phi \phi^{-1}(e_{i_\circ \rho}\iota (p)e_{j_\circ \sigma}) \\
	&=\sum_{c\in \mathcal C, \, c=  b p, \,  b \in G( a)} \lambda_c \mu_{ b \rho\sigma}(1\otimes e_\rho)\iota(p)(1\otimes e_\sigma).
	\end{align*}
	 
	It remains to determine the coefficients $\mu_{ b \rho\sigma}$ precisely.
	To do this, we can use the formula of Lemma~\ref{lem:iota(arrow)}. From the above computations we have that
	\[
	\mu_{b\rho\sigma}\tilde a_{\rho\sigma}= (j_\circ, \sigma)\phi^{-1}\iota(b)(i_\circ, \rho)
	\]
	for all $b\in G(a)$. On the other hand, if we set $b= g^{-1}(a)$, we have by Lemma~\ref{lem:iota(arrow)}:
	\[
	(j_\circ, \sigma)\phi^{-1}\iota(b)(i_\circ, \rho) =  \sigma\left(g\kappa_{\target(b)}^{-1}\right)\rho\left(g^{-1}\kappa_{\source(a)}^{-1}\kappa_{\source(b)}\right)\tilde a_{\rho\sigma},
	\]
	so we obtain
	\[
	\mu_{g^{-1}(a)\rho\sigma}  = \sigma\left(g\kappa_{\target(g^{-1}(a))}^{-1}\right)\rho\left(g^{-1}\kappa_{\source(a)}^{-1}\kappa_{\source(g^{-1}(a))}\right).
	\]

	Let now $L_ a\subseteq G$ be a set of representatives of $G/G_{i_\circ j_\circ}$ such that $g^{-1}( a)$ is an arrow for every $g\in L_ a$ (this exists by Lemma~\ref{lem:perm}). Then we have (recall that in the formula we got $b$ is assumed to be an arrow)
	\begin{align*}
	\Phi d (\tilde  a_{\rho \sigma}^*) &=
	\sum_{c\in \mathcal C, \, c=  b p, \,  b \in G( a)} \lambda_c \mu_{ b \rho\sigma}(1\otimes e_\rho)\iota(p)(1\otimes e_\sigma)\\
	&=\sum_{g\in L_ a}\sum_{c\in \mathcal C, \, c=  b p, \,  g(b)=a}\lambda_c \mu_{ b \rho\sigma}(1\otimes e_\rho)\iota(p)(1\otimes e_\sigma)\\
	&=\sum_{g\in L_ a}\sum_{c\in \mathcal C, \, g(c)= aq}\lambda_c \mu_{g^{-1}(a)\rho\sigma} (1\otimes e_\rho)\iota(g^{-1}(q))(1\otimes e_\sigma)\\
	&=\sum_{g\in L_ a}\sum_{c\in \mathcal C, \, g(c)= aq}\lambda_c \mu_{g^{-1}(a)\rho\sigma} \left(1\otimes e_\rho\kappa_{\source(g^{-1}(a))}\right)(g^{-1}(q)\otimes 1)\left(1\otimes e_\sigma \kappa_{\target(g^{-1}(a))}^{-1}\right)\\
	&=\sum_{g\in L_ a}\sum_{c\in \mathcal C, \, g(c)= aq}\lambda_c \mu_{g^{-1}(a)\rho\sigma} \left(1\otimes e_\rho\kappa_{\source(g^{-1}(a))}g^{-1}\right)(q\otimes 1)\left(1\otimes ge_\sigma \kappa_{\target(g^{-1}(a))}^{-1}\right)\\
	&=\sum_{g\in L_ a}\sum_{c\in \mathcal C, \, g(c)= aq}\lambda_c \mu_{g^{-1}(a)\rho\sigma} \left(1\otimes e_\rho\kappa_{\source(g^{-1}(a))}g^{-1}\kappa_{\source(a)}^{-1}\right)(1\otimes \kappa_{\source(a)})(q\otimes 1)\left(1\otimes ge_\sigma \kappa_{\target(g^{-1}(a))}^{-1}\right)\\
	&=\sum_{g\in L_ a}\sum_{c\in \mathcal C, \, g(c)= aq}\lambda_c (1\otimes e_\rho)(1\otimes \kappa_{\source(a)})(q\otimes 1)\left(1\otimes e_\sigma \right)\\
	&=\sum_{g\in L_ a}\sum_{c\in \mathcal C, \, g(c)= aq}\lambda_c (1\otimes e_\rho)\iota(q)(1\otimes e_\sigma),
	\end{align*}
	where, in the sixth equality, we have used that for a character $\tau$ of a group $H$, we have $e_\tau h= \tau(h^{-1})e_\tau$ for all $h\in H$. 
	
	To continue, we remark that $q= g(p)$ need not be a path. However, $g(c)= m_{g, c} d$, where $m_{g, c}\in k$ and $d\in \mathcal C$. From the assumption that $g(W) = W$ we obtain
	\[\lambda_c m_{g,c}d = g(\lambda_c c) = \lambda_d d,
	\]
     so that 
     \begin{align*}
     \Phi d (\tilde  a_{\rho \sigma}^*) &=
     \sum_{g\in L_ a}\sum_{c\in \mathcal C, \, g(c)= aq}\lambda_c m_{g,c} (1\otimes e_\rho)\iota\left(\frac{q}{m_{g,c}}\right)(1\otimes e_\sigma)\\
     &= \sum_{g\in L_ a}\sum_{c\in \mathcal C, \, g(c)=m_{g, c} d, \, d= aq' }\lambda_d (1\otimes e_\rho)\iota\left(q'\right)(1\otimes e_\sigma)\\
     &= \frac{|G|}{|G_{i_\circ j_\circ}|} \sum_{d\in \mathcal C, \, d= a q'} 
	\lambda_d	(1\otimes e_\rho)\iota(q')(1\otimes e_\sigma) \\
	&= d\Phi (\tilde  a_{\rho \sigma}^{*}).
     \end{align*}
	This concludes the proof that $\Phi$ commutes with $d$ in degree -1.
	
	For degree -2, we need to show that 
	$d\Phi (t_{(i_\circ, \rho)}) = \Phi d(t_{(i_\circ, \rho)})$. 
	Using the Leibniz rule, we have 
	\begin{align*}
	d\Phi (t_{(i_\circ, \rho)}) &=
	\frac{|G|}{|G_{i_\circ}|} d\left ( (1\otimes e_\rho)\iota(t_{i_\circ})(1\otimes e_\rho)  \right) \\
	&= 	\frac{|G|}{|G_{i_\circ}|} (1\otimes e_\rho)(d(t_{i_\circ})\otimes 1)(1\otimes e_\rho)  \\
	&= 	\frac{|G|}{|G_{i_\circ}|} (1\otimes e_\rho)
	\left(\sum_{ a\in Q_1}i_\circ [ a,  a^*]i_\circ\otimes 1\right)(1\otimes e_\rho)  \\
	&= (1) + (2)+ (3)+ (4 )+ (5), 
	\end{align*}
   where 
   \begin{align*}
   (1) &= \frac{|G|}{|G_{i_\circ}|} \sum_{ a: j\to i_\circ, \, j\not \in G(i_\circ)} (1\otimes e_\rho)( a a^*\otimes 1)(1\otimes e_\rho), \\
   (2) &= - \frac{|G|}{|G_{i_\circ}|} \sum_{ a: i_\circ\to j, \, j\not \in G(i_\circ)} (1\otimes e_\rho)( a^* a\otimes 1)(1\otimes e_\rho), \\
   (3) &= \frac{|G|}{|G_{i_\circ}|} \sum_{ a: i\to i_\circ, \, i\in G(i_\circ)\setminus \{i_\circ\}} (1\otimes e_\rho)( a a^*\otimes 1)(1\otimes e_\rho), \\
   (4) &= - \frac{|G|}{|G_{i_\circ}|} \sum_{ a: i_\circ\to i, \, i \in G(i_\circ)\setminus \{i_\circ\}} (1\otimes e_\rho)( a^* a\otimes 1)(1\otimes e_\rho), \\
   (5) &= \frac{|G|}{|G_{i_\circ}|} \sum_{ a: i_\circ\to i_\circ} (1\otimes e_\rho)([ a,  a^*]\otimes 1)(1\otimes e_\rho).
   \end{align*}
	
	For an arrow $ a:j\to i_\circ$, let $L_{ a, i_\circ}$ be a set of representatives of $G_{i_\circ}/ G_{i_\circ j_\circ}$ mapping $ a$ to arrows (this exists by Lemma~\ref{lem:perm}). Then 
	\begin{align*}
	(1) &= 
	\sum_{ a\in D(j_\circ, i_\circ), \, i_\circ \neq j_\circ} \sum_{g\in L_{ a, i_\circ}} \frac{|G|}{|G_{i_\circ}|}(1\otimes e_\rho )(g( a)g( a)^*\otimes 1)(1\otimes e_\rho)  \\
		&=\sum_{a\in D(j_\circ, i_\circ), \, i_\circ \neq j_\circ} \sum_{g\in L_{ a, i_\circ}} \frac{|G|}{|G_{i_\circ}|}(1\otimes e_\rho )(1\otimes g)( a a^*\otimes 1)(1\otimes g^{-1})(1\otimes e_\rho)  \\
	&= 	\sum_{a\in D(j_\circ, i_\circ), \, i_\circ \neq j_\circ} \sum_{g\in L_{ a, i_\circ}} \frac{|G|}{|G_{i_\circ}|}(1\otimes e_\rho)( a a^*\otimes 1)(1\otimes e_\rho) \\
	&= \frac{|G|}{|G_{i_\circ j_\circ}|}\sum_{a\in D(j_\circ, i_\circ), \, i_\circ \neq j_\circ}(1\otimes e_\rho)\iota( a a^*)(1\otimes e_\rho).
	\end{align*}
	Similarly we have 
	\begin{align*}
	(3) &= \frac{|G|}{|G_{i_\circ}|}\sum_{a\in D(i_\circ, i_\circ), \, \source(a)\neq i_\circ}(1\otimes e_\rho)\iota( a a^*)(1\otimes e_\rho).
	\end{align*}
	Let now $ a:i\to j$ be an arrow, with $i\in G(i_\circ)$. Then choose $N_{ a, i_\circ}\subseteq G$ to be a maximal set mapping $ a$ to distinct arrows and $i$ to $i_\circ$. Thus for $g\in N_{ a, i_\circ}$, $g = g\kappa_i^{-1}\kappa_i$, which means that $N_{ a, i_\circ}$ is a shift by $\kappa_i$ of a set of representatives of $G_{i_\circ}/G_{i_\circ j_\circ}$. 
	We have
	\begin{align*}
	(2) &= 
		-\frac{|G|}{|G_{i_\circ}|}\sum_{ a: i\to j, \, j\not \in G(i_\circ), \, i\in G(i_\circ)}  (1\otimes e_\rho )(i_\circ  a^* a i_\circ\otimes 1)(1\otimes e_\rho)  \\
		&= -\frac{|G|}{|G_{i_\circ}|}\sum_{a\in D(i_\circ, j_\circ), \, i_\circ \neq j_\circ} \sum_{g\in N_{ a, i_\circ}}  (1\otimes e_\rho )(i_\circ g( a)^*g( a) i_\circ\otimes 1)(1\otimes e_\rho)  \\
		&= -\frac{|G|}{|G_{i_\circ}|}\sum_{a\in D(i_\circ, j_\circ), \, i_\circ \neq j_\circ} \sum_{g\in N_{ a, i_\circ}}  \left(1\otimes e_\rho g\kappa_{\source(a)}^{-1} \right)(1\otimes \kappa_{\source(a)})(  a^* a\otimes 1)\left(1\otimes \kappa_{\source(a)}^{-1}\right)\left(1\otimes e_\rho g^{-1}\kappa_{\source(a)}\right)  \\
		&= -\frac{|G|}{|G_{i_\circ j_\circ}|}\sum_{a\in D(i_\circ, j_\circ), \, i_\circ \neq j_\circ} (1\otimes e_\rho )\iota( a^* a)(1\otimes e_\rho ). 
	\end{align*}
	In the same way, we get
	\begin{align*}
	(4) &= 
	-\frac{|G|}{|G_{i_\circ}|}\sum_{a\in D(i_\circ, i_\circ), \, \source(a)\neq i_\circ} (1\otimes e_\rho )\iota( a^* a)(1\otimes e_\rho ). 
	\end{align*}
	Finally, it follows immediately from the definition of $\iota$ that
		\begin{align*}
	(5) &= 
	\frac{|G|}{|G_{i_\circ}|}\sum_{ a: i_\circ\to i_\circ} (1\otimes e_\rho )\iota([ a, a^*])(1\otimes e_\rho ). 
	\end{align*}
	In fact, we see that
		\begin{align*}
	(3)+ (4)+ (5) &= 
	\frac{|G|}{|G_{i_\circ}|}\sum_{a\in D(i_\circ, i_\circ)} (1\otimes e_\rho )\iota(\left[ a,  a^*\right])(1\otimes e_\rho ). 
	\end{align*}
	
	On the other hand, we can compute
		\begin{align*}
		\Phi d(t_{(i_\circ, \rho)})&= 
		\Phi \sum_{\tilde  a_{\tau \omega}\in (Q_G)_1} (i_\circ, \rho)\left[\tilde a_{\tau \omega}, \tilde a_{\tau \omega}^*\right] (i_\circ, \rho) \\
		&= (A)+ (B )+ (C),
	\end{align*}
	where
	\begin{align*}
	(A) &= -\Phi \sum_{ a\in D(i_\circ, j_\circ), \, i_\circ \neq j_\circ}\sum_{\sigma\in G_{j_\circ}^\vee} \tilde a_{\rho \sigma}^*\tilde  a_{\rho \sigma},\\
	(B) &= \Phi \sum_{ a\in D(j_\circ, i_\circ), \, i_\circ \neq j_\circ}\sum_{\sigma\in G_{j_\circ}^\vee} \tilde a_{ \sigma\rho}\tilde  a_{\sigma\rho}^*, \\
	(C) &= \Phi \sum_{ a\in D(i_\circ, i_\circ)}\sum_{\sigma\in G_{i_\circ}^\vee} \left[\tilde a_{\rho\sigma},\tilde  a_{\rho\sigma}^*\right]. 
	\end{align*}
	Now 
	\begin{align*}
	(A)&=-\frac{|G|}{|G_{i_\circ j_\circ}|}\sum_{a\in D(i_\circ, j_\circ), \, i_\circ \neq j_\circ}\sum_{\sigma\in G_{j_\circ}^\vee}
	(1\otimes e_\rho)\iota( a^*)(1\otimes e_\sigma)\iota( a)(1\otimes e_\rho)\\
	&= (2).
	\end{align*}
	In the same way, $(B) = (1)$ and $(C) = (3)+(4)+(5)$, which concludes the proof.
\end{proof}

A consequence of Theorem~\ref{thm:main} is that we obtain functors between the generalized cluster categories of~$(Q,W)$ and of~$(Q_G, W_G)$.  For a dg algebra~$A$, let~$\cD A$ be its derived category,~$\perf A$ be its perfect derived category (that is, the smallest triangulated subcategory of~$\cD A$ containing~$A$ and stable under taking direct summands) and~$\cD_{\fd}A$ be the subcategory of~$\cD A$ of all objects whose cohomology has finite total dimension (a good reference for derived categories of dg algebras is \cite{K94}).

By~\cite[Proposition 2.4]{AP17}, the canonical inclusion of~$\Gamma_{Q,W}$ in~$\Gamma_{Q,W}G$ induces a functor

 \[
  ?\otimes_{\Gamma_{Q,W}}^L \Gamma_{Q,W} G : \cD \Gamma_{Q,W} \to \cD \Gamma_{Q,W}G
 \]
which restricts to functors
 \[
  ?\otimes_{\Gamma_{Q,W}}^L \Gamma_{Q,W} G : \perf \Gamma_{Q,W} \to \perf \Gamma_{Q,W}G
 \]
 and
  \[
  ?\otimes_{\Gamma_{Q,W}}^L \Gamma_{Q,W} G : \cD_{\fd} \Gamma_{Q,W} \to \cD_{\fd} \Gamma_{Q,W}G.
 \] 
 Similar statements (with arrows reversed) are true for the adjoint functor 
 \[
  \operatorname{RHom}_{\Gamma_{Q,W}G}(\Gamma_{Q,W} G, ?) : \cD \Gamma_{Q,W}G \to \cD \Gamma_{Q,W}.
 \]
Since the generalized cluster category of~$(Q,W)$ is defined to be the Verdier quotient~\[\cC_{Q,W} = \perf \Gamma_{Q,W} / \cD_{\fd} \Gamma_{Q,W},\] we obtain the following generalization of~\cite[Corollary 2.8]{AP17}.

\begin{cor}
 Keep the hypotheses of Theorem~\ref{thm:main}.  Then the pair of adjoint functors between~$\cD \Gamma_{Q,W}$ and~$\cD \Gamma_{Q,W}G$ described above induce functors 
 
 \[
  F : \cC_{Q,W} \to \cC_{Q_G, W_G} \quad \textrm{and} \quad F': \cC_{Q_G, W_G} \to \cC_{Q,W}.
 \] 
\end{cor}
\begin{proof}
 This follows from the above discussion and from Theorem~\ref{thm:main}, using the fact that~$\overline{e} \Gamma_{Q,W} G \overline{e}$ and~$\Gamma_{Q,W}$ are Morita equivalent (by Proposition~\ref{prop:eGammae}).
\end{proof}

\section{Example}\label{sec:example}
In this section we compute a detailed example to showcase our construction. For this example we have to assume that $3\neq 0$ in $k$.
Let $G= \langle g, h \,|\, gh = hg, \, g^3= h^3 =1\rangle \cong \Z/3\Z \times \Z/3\Z$. 
Let $Q$ be the quiver of Figure~\ref{fig:ex_quiver}. 
\begin{figure}[h!]
	\[
	\begin{tikzcd}[arrow style=tikz,>=stealth,row sep=5em]
	& & i_1  \arrow[out=60,in=120,loop,swap,"x_1"]
	\arrow[dl]\arrow[dr]\arrow[dd, bend left = 20pt]
	& & \\
	& j_1 \arrow[dr, swap, "y_1"]
	& & j_3 \arrow[ll, swap, "y_3"]
	& & \\
	i_2 \arrow[out=180,in=240,loop,swap,"x_2"]
	\arrow[ur]\arrow[urrr, bend left = 20pt]\arrow[rr]
	& & j_2\arrow[ur,swap, "y_2"]
	& & i_3\arrow[out=300,in=0,loop,swap,"x_3"]
	\arrow[ul]\arrow[ll]\arrow[ulll, bend left = 20pt]
	\end{tikzcd}
	\]
	\caption{The quiver $Q$.}
		\label{fig:ex_quiver}
\end{figure}
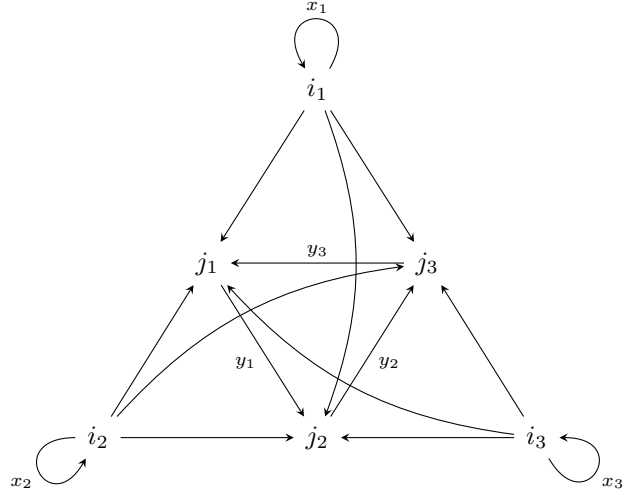
For clarity of notation, we have not given a name to all the arrows.

We will now define an action of $G$ on $Q$. Let $\zeta$ be a primitive third root of unity. 
We let $G$ act on $Q_0$ by $g(j_l) = j_l, g(i_l ) = i_{l+1}, h(i_l) = i_l$, and $h(j_l) = j_{l+1}$. We set $g(x_l) = x_{l+1}$ and $h(x_l) = \zeta x_l$. We let $G$ map the other arrows to arrows in the obvious way. 
If we set $W= y_3y_2y_1 + x_1^3 + x_2^3 + x_3^3$, then $W$ is a potential on $Q$, and we get that $g(W) = h(W) = W$ in $P_Q$. Thus the action of $G$ on $(Q, W)$ is within our setup, and we can apply our construction. 

We need to make some choices.
 Let $\tilde{I} = \{i_1, j_1\}$, and let us write $i$ and $j$ for $i_1, j_1$ for simplicity. We have $G_i = \langle h\rangle$ and $G_j = \langle g\rangle$. Let us denote by $\omega$ (respectively $\tau$) the representations of $G_i$ (respectively~$G_j$) sending $h$ (respectively~$g$) to $\zeta$. We write $\on{tr}$ for the trivial representation of both $G_i$ and $G_j$. By Notation~\ref{not:q0}, the quiver $Q_G$ has vertices as in Figure~\ref{fig:vertices}.
 
 \begin{figure}[h!]
\[
\begin{tikzcd}[arrow style=tikz,>=stealth,row sep=6em]
(i, \on{tr}) & & (j, \on{tr}) && (i, \omega) \\
& (j, \tau) & & (j, \tau^2) \\
 &  & (i, \omega^2) & & 
\end{tikzcd}
\]
\caption{The vertices of $Q_G$.}
\label{fig:vertices}
\end{figure}
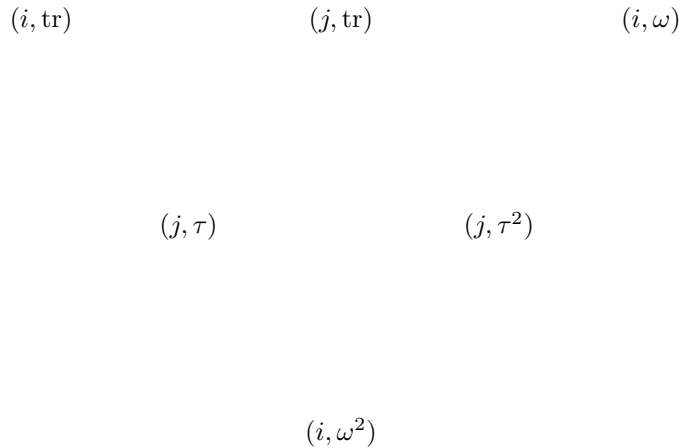

Let us now determine the arrows of $Q_G$, as in Lemma~\ref{lem:q1}.
Let us look at arrows of the form $i_l\to j_m$. Since $G_i\cap G_j = \{1\}$, we get that there are arrows in $Q_G$ from $(i, \rho)$ to $(j, \sigma)$ for all $\rho, \sigma$. 

Let us look at the arrows $x_l$. Since $\chi_{x_l} = \omega$, there is an arrow in $Q_G$ from $(i, \rho)$ to $(i, \sigma)$ whenever $\rho = \sigma \omega$. 

Finally, let us consider the arrows $y_l$. We have $\chi_{y_l} = \on{tr}$, so we get a loop $(j, \sigma)\to (j, \sigma)$ for all $\sigma$. The quiver $Q_G$ is thus as in Figure~\ref{fig:QG}.

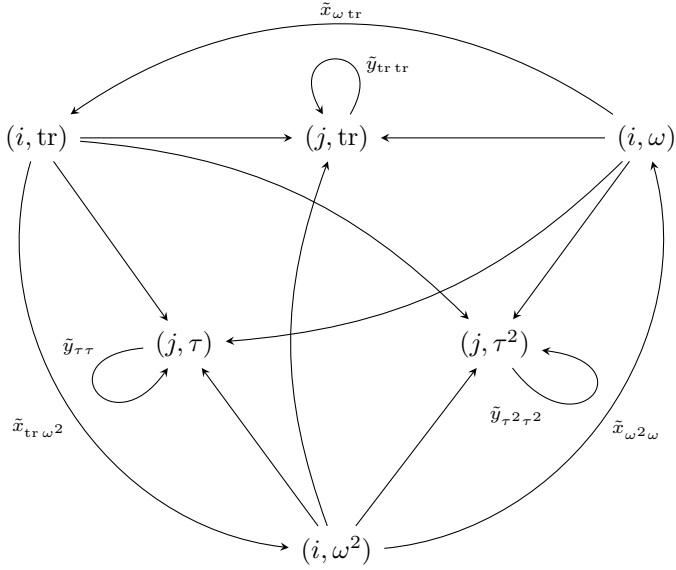
\begin{figure}[h!]
\[
\begin{tikzcd}[arrow style=tikz,>=stealth,row sep=6em]
(i, \on{tr}) \arrow[rr]\arrow[rrrd, bend left=20pt]\arrow[rd]
\arrow[rrdd, bend right = 50pt, swap, "\tilde x_{\on{tr}\omega^2}"]
& & (j, \on{tr}) \arrow[out= 60, in = 120, loop, "\tilde y_{\on{tr}\on{tr}}", pos = .2, swap]
&& (i, \omega)\arrow[ll]\arrow[ld]\arrow[llld, bend left = 20pt]
\arrow[llll, bend right = 35pt, swap, "\tilde x_{\omega\on{tr}}"]
 \\
& (j, \tau)\arrow[out = 185, in = 235, loop,"\tilde y_{\tau\tau}", pos = .2, swap]
 & & (j, \tau^2) \arrow[out = 307, in = 353, loop,"\tilde y_{\tau^2\tau^2}", pos = .2, swap]
 \\
&  & (i, \omega^2) \arrow[ul]\arrow[ur]\arrow[uu, bend left = 20pt]
\arrow[uurr, bend right = 50 pt, swap, "\tilde x_{\omega^2\omega}"]
& & 
\end{tikzcd}
\]
\caption{The quiver $Q_G$.}
\label{fig:QG}
\end{figure}

Let us now compute the potential $W_G$, as in Section~\ref{sec:mainresult}. We recall that, by Remark~\ref{rem:lift}, it is enough to compute $\cyc\phi^{-1}\iota (y_3y_2y_1 + x_1^3 + x_2^3 + x_3^3)$.
The term $\cyc\phi^{-1}\iota (y_3 y_2 y_1)$ will consist of a linear combination of the three cycles $\tilde{y}_{\sigma \sigma}^3$, $\sigma = \on{tr}, \tau, \tau^2$, with some coefficients. We can determine these coefficients with the help of the second formula of Proposition~\ref{prop:iota(path)}. 

First, observe that $y_3$ is a distinguished arrow, and that $h(y_2)= y_3$ and $h^{-1}(y_1) = y_3$. Moreover, $\kappa_{j_3} = h$. By the formula, the coefficient of $\tilde{y}_{\sigma \sigma}^3$ is $$\sigma(1\cdot h\cdot h^{-1})\sigma(h^{-1}\cdot h^{-1}\cdot h^{-1})\sigma(h\cdot 1\cdot h^{-1}) = \sigma(h^{-3}) = 1$$
for all $\sigma \in \{\on{tr}, \tau, \tau^2\}$.  

Let us look at the term $\cyc\phi^{-1}\iota(x_1^3)$. This will be equal to the cycle $\tilde{x}_{\omega \on{tr}}\tilde x_{\omega^2 \omega} \tilde x_{\on{tr}\omega^2}$ times a coefficient. The arrow $x_1$ is distinguished and $\kappa_{i_1} = 1$, therefore the coefficient is simply
$$
\on{tr}(1\cdot 1\cdot 1)\omega(1\cdot 1\cdot 1)\omega^2(1\cdot 1\cdot 1) = 1.$$
Ultimately, we have determined that 
$$W_G = \tilde{x}_{\omega \on{tr}}\tilde x_{\omega^2 \omega} \tilde x_{\on{tr}\omega^2} + \tilde{y}_{\on{tr}, \on{tr}}^3 + \tilde{y}_{\tau\tau}^3 + \tilde{y}_{\tau^2\tau^2}^3.$$
We remark that in this example the quiver with potential $(Q_G, W_G)$ is isomorphic to the opposite of $(Q, W)$. 
By Theorem~\ref{thm:main}, the Ginzburg dg algebra $\Gamma_{Q_G, W_G}$ is Morita equivalent to the skew group dg algebra $\Gamma_{Q, W}G$.

\section{Index of notations} \label{sec:index_nota}

Since this article is quite technical and required us to set up a great deal of notation, we include here a table of symbols. For each symbol we write a short and imprecise description, as well as redirect the reader to where in the text the precise definition can be found.\\

\begin{tabular}{|c|p{10cm}|p{2.1cm}|}
	\hline
		&&\\[-1em]
	Symbol & Description & Reference\\ \hline\hline
	&&\\[-1em]
	$\widehat{k Q}_{cyc}$ & The complete algebra of cyclic paths of $Q$. & 
	Section~\ref{subsec:qp}
	\\ \hline
	\\[-1em]
	$\sh$ & The shuffle function for cyclic paths. &
	Section~\ref{subsec:qp}
	\\ \hline
&&\\[-1em]
	$e_\rho, e_{\rho|_H}$ & Idempotents of a group algebra corresponding to the 
	character $\rho$. &
	Section~\ref{subsec:sga}
	\\ \hline
	&&\\[-1em]
	$G(i), G_i$ & The orbit and stabilizer of $i$ under the action of $G$. &
	Section~\ref{subsec:sga}
	\\ \hline
	&&\\[-1em]
	$G_{ij}$ & The common stabilizer of $i$ and $j$. &
	Section~\ref{subsec:sga}
	\\ \hline
	&&\\[-1em]
	$\tilde I$ & A chosen set of representatives of $Q_0$ under the action of $G$. &
	Notation~\ref{not:itilde}
	\\ \hline
	&&\\[-1em]
	$i_\circ, j_\circ$ & Vertices in $\tilde{I}$. &
	Notation~\ref{not:itilde}
	\\ \hline
	&&\\[-1em]
	$M, MG, M_{ij}$ & Vector spaces generated by arrows. &
	Notation~\ref{not:M}
	\\ \hline
	&&\\[-1em]
	$\chi_a$ & The character of $G_{\source(a)\target(a)}$ such that $g(a)= \chi_a(g)a$ for $g \in G_{\source(a)\target(a)}$. &
	Notation~\ref{not:chi}
	\\ \hline
	&&\\
	$Q_G$ & The quiver of the skew group algebra $(kQ)G$. &
	Notation~\ref{not:q0}, Lemma~\ref{lem:q1}
	\\ \hline
	&&\\[-1em]
	$\bar e$ & A chosen idempotent of $(kQ)G$ such that $kQ_G\cong \bar e(kQ)G\bar e$. &
	Notation~\ref{not:ebar}
	\\ \hline
	&&\\[-1em]
	
	$\kappa_i$ & A chosen element of $G$ such that $\kappa_i(i)\in \tilde I$.&
	Notation~\ref{not:kappa}
	\\ \hline
	&&\\[-1em]
	$R_{i_\circ j_\circ}$ & A chosen set of representatives of $G(i_\circ)$ under the action of $G_{j_\circ}$. &
	Notation~\ref{not:distinguished}
	\\ \hline

	&&\\[-1em]
	$D(i_\circ, j_\circ), D$ & Chosen sets of distinguished arrows, used to describe $(Q_G)_1$. &
	Notation~\ref{not:distinguished}
	\\ \hline
	&&\\[-1em]
	$\tilde a_{\rho \sigma}$ & The arrow (if it exists) in $Q_G$ corresponding to $a\in Q$.&
	Notation~\ref{not:atilde}
	\\ \hline
	&&\\
	$\phi$ & The isomorphism $kQ_G\to \bar e(kQ)G\bar e$ which we define.  &
	Definition~\ref{def:phivertices}, Definition~\ref{def:phiarrows}
	\\ \hline
	&&\\[-1em]
	$\iota$ & An almost-multiplicative function $kQ\to (kQ)G$ we define. &
	Notation~\ref{def:iota}
	\\ \hline
	&&\\[-1em]
	$\on{cyc}$ & A function that is zero on non-cycles and the identity on cycles.&
	Notation~\ref{not:cyc}
	\\ \hline
	&&\\[-1em]
	$W_G$ & The potential we define on $Q_G$. &
	Definition~\ref{def:wg}
	\\ \hline
	
\end{tabular}

\bibliographystyle{alpha}
\bibliography{Bibliography}

\end{document}